\documentclass[a4paper]{amsart}

\usepackage{epic, eepic, amsfonts, latexsym, amssymb, graphicx,
multicol, mathrsfs, color, amscd, verbatim, paralist,
xspace, url, euscript, stmaryrd,  amsmath, enumitem,
bbold, multirow, tikz, mathtools}
\usepackage[all,pdf,cmtip]{xy}

\usepackage[colorlinks, linkcolor=blue, citecolor=magenta, urlcolor=cyan]{hyperref}

\def\hat{\widehat}
\renewcommand\bar{\overline}

\def\PP{{\mathbb P}}

\def\hat{\widehat}

\def\red{\mathop{\rm red}\nolimits}

\def\im{\mathop{\rm im}\nolimits}

\def\GL{\mathop{\rm GL}\nolimits}

\def\Res{\mathop{\rm Res}\nolimits}
\def\Gor{\mathop{\rm Gor}\nolimits}

\def\rank{\mathop{\rm rank}\nolimits}
\def\Jac{\mathop{\rm Jac}\nolimits}
\def\jac{\mathop{\rm jac}\nolimits}

\def\Gr{\mathop{\rm Gr}\nolimits}

\newcommand{\bA}{\mathbf{A}}

\newcommand{\bbf}{\mathbf{f}}
\newcommand{\bfGor}{\mathbf{Gor}}
\def\Char{\mathop{\rm char}\nolimits}
\def\SSMat{\mathop{\rm SSMat}\nolimits}

\newcommand\co{\colon} 

\newtheorem{theorem}{THEOREM}[section]

\newtheorem{proposition}[theorem]{Proposition}

\theoremstyle{definition}

\newtheorem{prop}[theorem]{Proposition}

\newtheorem{lemma}[theorem]{Lemma}

\newtheorem{definition}[theorem]{Definition}
\theoremstyle{remark}
\newtheorem{remark}[theorem]{Remark}

\renewcommand{\theequation}{\arabic{section}.\arabic{equation}}

\makeatletter
\def\blfootnote{\xdef\@thefnmark{}\@footnotetext}
\makeatother

\begin{document}

\title[Smoothness of catalecticant schemes]{A combinatorial proof of\\\vspace{0.15cm} the smoothness of catalecticant schemes\\\vspace{0.15cm} associated to complete intersections}

\author[Isaev]{Alexander Isaev}

\address{Mathematical Sciences Institute\\
Australian National University\\
Acton, ACT 2601, Australia}
\email{alexander.isaev@anu.edu.au}

\maketitle

\thispagestyle{empty}

\pagestyle{myheadings}

\begin{abstract} For zero-dimensional complete intersections with homogeneous ideal generators of equal degrees over an algebraically closed field of characteristic zero, we give a combinatorial proof of the smoothness of the corresponding catalecticant schemes along an open subset of a particular irreducible component.
\end{abstract}

\section{Introduction}\label{intro}
\setcounter{equation}{0}

Catalecticant varieties and schemes were introduced by A.~Iarrobino and V.~Ka\-nev in their seminal monograph \cite{IK} in relation to the classical problem of representing a homogeneous form as a sum of powers of linear forms as well as related questions. Of special interest are the irreducible components and smoothness properties of such schemes; they have been extensively studied (see \cite[Chapter 4]{IK} and references therein for details). In the present paper, we focus on these topics for particular catalecticant schemes, which we call ${\mathbf V}$ and $\bfGor(T)$, associated to complete intersections with homogeneous ideal generators of equal degrees. Specifically, fix $n\ge 2$ and let $k[x_1, \ldots, x_n]_j$ denote the vector space of homogeneous forms of degree $j$ in $x_1,\dots,x_n$ over an algebraically closed field $k$ of characteristic zero. The schemes ${\mathbf V}$ and $\bfGor(T)$ then arise from considering zero-dimensional complete intersection $k$-algebras of the form
$$
M(\bbf) := k[x_1, \ldots, x_n] /(f_1, \ldots, f_n),
$$
where $\bbf=(f_1,\dots,f_n)$ is an $n$-tuple of elements of $k[x_1, \ldots, x_n]_d$ with fixed $d\ge 2$  (see Section \ref{schemes} for details).

Next, let $U_{\Res} \subset k[y_1, \ldots, y_n]_{n(d-1)}$ be the subset of each of ${\mathbf V}$ and $\bfGor(T)$ defined as the locus of forms $F$ such that the subspace $F^{\perp} \cap k[x_1, \ldots, x_n]_d$ is $n$-dimensional and has a basis with nonvanishing resultant, where 
$$
F^{\perp} := \left\{h\in k[x_1, \ldots, x_n]\,\Big| \, h\left(\frac{\partial}{\partial y_1}, \ldots, \frac{\partial}{\partial y_n}\right)F(y_1, \ldots, y_n) = 0 \right\}
$$
is the annihilator of $F$. In \cite[ Theorem 4.17]{IK} it was shown that $\bfGor(T)$ has an irreducible component containing $U_{\Res}$ as a dense subset and the dimension of this component was found. Furthermore, the smoothness of $\bfGor(T)$ at every point of $U_{\Res}$ follows from \cite[Theorem 1.4]{K} (cf.~\cite[p.~117--118]{IK}). On the other hand, analogous facts regarding ${\mathbf V}$ appear to be only known in the cases (i) $n=3$, $d\ge 3$, (ii) $n=4$, $d=2,3$, (iii) $n=5$, $d=2$ (see \cite[Theorem 4.19 and Corollary 4.18]{IK}), and one of the aims of the present paper is to bring results on ${\mathbf V}$ in line with those on $\bfGor(T)$.

In this article, we refine and extend Theorems 4.17 and 4.19 of \cite{IK}. First of all, in Theorem \ref{main8} obtained in Section \ref{schemes} we show that the set $U_{\Res}$ is open (not just dense) in an irreducible component of each of ${\mathbf V}$, $\bfGor(T)$ for all $n,d$ and explicitly describe the closed complement to $U_{\Res}$. As the proof of Theorem 4.17 in \cite{IK} is quite brief, we also provide an alternative derivation---with full details---of the dimension formula for $U_{\Res}$. One of the key elements of the proof of Theorem \ref{main8} is representing $U_{\Res}$ as the image of a certain morphism, called $\bA$, which is introduced and discussed in Section \ref{morphism}. This morphism assigns to every $n$-tuple of forms $\bbf\in k[x_1, \ldots, x_n]_d^{\oplus n}$ with nonzero resultant the so-called {\it associated form}\, lying in $k[y_1, \ldots, y_n]_{n(d-1)}$ and can be interpreted as mapping $\bbf$ to a particular homogeneous Macaulay inverse system of the standard graded Artinian Gorenstein algebra $M(\bbf)$. As explained in \cite{AI}, the morphism $\bA$ is also of interest in relation to a conjecture (not discussed here) linking complex singularity theory with classical invariant theory. We note that settling the conjecture requires studying ${\mathbf V}$ and $\bfGor(T)$, which explains our interest in these schemes.

The main content of the paper is an argument that establishes the smoothness of both ${\mathbf V}$ and $\bfGor(T)$ simultaneously at every point of $U_{\Res}$ for all $n,d$ (see Theorem \ref{main888} in Section \ref{results}). It extends the proof of Theorem 4.19 in \cite{IK}, which only works in special cases (i)--(iii) mentioned earlier. Our argument is direct and does not depend on results and techniques of \cite{K}. In particular, it requires neither Laudal's description of hulls nor results on the tangent and obstruction space for the scheme $GradAlg(H)$ (cf.~\cite[p.~610]{K}). Apart from the formulas for the dimensions of the tangent spaces to ${\mathbf V}$ and $\bfGor(T)$ obtained in \cite[Theorems 3.2, 3.9]{IK}, it relies just on the Koszul resolution of the algebra $M(\bbf)$. The novelty of the proof of Theorem \ref{main888} is that it proceeds by dimension count and is almost entirely combinatorial. In particular, we obtain a number of combinatorial identities that are independently interesting. We also note that, although the field $k$ is assumed to have zero characteristic, our arguments are easy to generalize to the case $\Char(k)>n(d-1)$, with $n(d-1)$ being the socle degree of $M(\bbf)$.

{\bf Acknowledgements.} Part of this work was done during the author's visit to the Bar-Ilan University, which we thank for its hospitality. We are also grateful to Marko Riedel for his help with establishing one of the combinatorial identities as specified in the appendix. Special thanks go to the referees for their thorough reading of the manuscript and for suggesting, in particular, an alternative proof of identity (\ref{desiredestim}) (see Remark \ref{alternpr} for details). We acknowledge the support of the Australian Research Council by way of Discovery Project DP140100296.

\section{Associated forms and Macaulay inverse systems}\label{morphism}
\setcounter{equation}{0}

In this section we introduce the so-called associated forms and the corresponding morphism, which will be useful for our study of the catalecticant schemes in the next section. What follows is an abridged version of the exposition given in \cite[Section 2]{AI}.

Let $k$ be an algebraically closed field, and we assume for simplicity that its characteristic is zero. Fix $n\ge 2$ and for any nonnegative integer $j$ define $k[x_1, \ldots, x_n]_j$ to be the vector space of homogeneous forms of degree $j$ in $x_1,\dots,x_n$ over $k$. Clearly, one has $k[x_1,\dots,x_n]=\oplus_{j=0}^{\infty}k[x_1, \ldots, x_n]_j$. Next, fix $d\ge 2$ and consider the vector space $k[x_1, \ldots, x_n]_d^{\oplus n}$ of $n$-tuples $\bbf = (f_1, \ldots, f_n)$ of forms of degree $d$.  Recall that the resultant $\Res$ on the space $k[x_1, \ldots, x_n]_d^{\oplus n}$ is a form with the property that $\Res(\bbf) \neq 0$ if and only if $f_1, \ldots, f_n$ have no common zeroes away from the origin (see, e.g.,  \cite[Chapter 13]{GKZ}).

For $\bbf = (f_1, \ldots, f_n) \in k[x_1, \ldots, x_n]_d^{\oplus n}$, we now introduce the algebra
$$
M(\bbf) := k[x_1, \ldots, x_n] / (f_1, \ldots, f_n)
$$
and recall a well-known lemma (see, e.g., \cite[Lemma 2.4]{AI} and \cite[p.~187]{SS}):

\begin{lemma}\label{fourconds} \it The following statements are equivalent:
\begin{enumerate}
	\item[\rm (1)] the resultant $\Res(\bbf)$ is nonzero;
	\item[\rm (2)] the algebra $M(\bbf)$ has finite vector space dimension;
	\item[\rm (3)] the morphism $\bbf \co {\mathbb A}^n(k) \to {\mathbb A}^n(k)$ is finite;
	\item[\rm (4)] the $n$-tuple $\bbf$ is a homogeneous system of parameters of $k[x_1,\dots,x_n]$, i.e., the Krull dimension of $M(\bbf)$ is $0$.
\end{enumerate}
If the above conditions are satisfied, then $M(\bbf)$ is a local standard graded complete intersection algebra whose socle is generated in degree $n(d-1)$ by the image\linebreak $\bar{\jac(\bbf)} \in M(\bbf)$ of the Jacobian $\jac(\bbf):= \det \Jac(\bbf)$, where $\Jac(\bbf)$ is the Jacobian matrix $\big({\partial f_i}/{\partial x_j} \big)_{i,j}$.\end{lemma}

\begin{remark}\label{hilbert}
As we pointed out in Lemma \ref{fourconds}, the algebra $M(\bbf)$ has a natural standard grading: $M(\bbf) = \bigoplus_{i=0}^{\infty} M(\bbf)_i$.  It is well-known (see, e.g., \cite[Corollary 3.3]{S}) that the corresponding Hilbert function $H(u):=\sum_{i=0}^{\infty}t_i\,u^i$, with $t_i:=\dim_k{M}({\mathbf f})_i$, is given by
\begin{equation}
H(u)=(u^{d-1}+\dots+u+1)^n.\label{hf}
\end{equation}
\end{remark}

Next, we let $(k[x_1, \ldots, x_n]_d^{\oplus n})_{\Res}$ be the affine open subvariety of $k[x_1, \ldots, x_n]_d^{\oplus n}$ that consists of all $n$-tuples of forms with nonzero resultant. We now define the {\it associated form}\, $\bA(\bbf) \in k[y_1, \ldots, y_n]_{n(d-1)}$ of $\bbf=(f_1,\dots,f_n) \in (k[x_1, \ldots, x_n]_d^{\oplus n})_{\Res}$ by the formula
$$(y_1 \bar{x}_1 + y_2 \bar{x}_2 + \cdots + y_n \bar{x}_n)^{n(d-1)} = \bA(\bbf)(y_1, \ldots, y_n) \cdot \bar{\jac(\bbf)},$$
where $\bar{x}_i \in M(\bbf)$ is the image of $x_i$. It is not hard to see that the induced map
$$
\bA \co (k[x_1, \ldots, x_n]_d^{\oplus n})_{\Res} \to k[y_1, \ldots, y_n]_{n(d-1)}, \quad \bbf \mapsto \bA(\bbf)
$$
is a morphism of affine varieties. In article \cite{AI} we studied $\bA$ in relation to a conjecture linking complex singularity theory with classical invariant theory.

The morphism $\bA$ is quite natural; in particular, it possesses an important equivariance property, which we will now state. First, notice that for any $j$ the group $\GL_n \times \GL_n$ acts on the vector space $k[x_1, \ldots, x_n]_j^{\oplus n}$ via
$$
((g_1, g_2)  {\mathbf f}) (x) := {\mathbf f} (x \cdot g_1^{-t}) \cdot g_2^{-1}\label{doubleaction}
$$
for $g_1,g_2 \in \GL_n$, $x:=(x_1,\dots,x_n)$ and $\bbf\in k[x_1, \ldots, x_n]_j^{\oplus n}$. Also, for any $\ell$ the group $\GL_n$ act on the space $k[y_1, \ldots, y_n]_{\ell}$ via
$$
(g f) (y) := f (y \cdot g^{-t})\label{doubleaction1}
$$
for $g \in \GL_n$, $y:=(y_1,\dots,y_n)$ and $f\in k[y_1, \ldots, y_n]_{\ell}$. We then have (see \cite[Lemma 2.7]{AI}):  

\begin{lemma}\label{L:equiv2} \it For every ${\mathbf f}\in (k[x_1, \ldots, x_n]_d^{\oplus n})_{\Res}$ and $g_1,g_2 \in \GL_n$ the following holds:
\begin{equation} \label{E:equiv2}
\displaystyle{\mathbf A}((g_1,g_2) \bbf)=\det(g_1 g_2) \cdot g_1^{-t} {\mathbf A}({\mathbf f}).
\end{equation}
\end{lemma} 

We will now interpret $\bA$ in different terms. Recall that the algebra $k[y_1, \ldots, y_n]$ is a $k[x_1, \ldots, x_n]$-module via differentiation:
\begin{equation}
(h \circ F) (y_1, \ldots, y_n) := h\left(\frac{\partial}{\partial y_1}, \ldots, \frac{\partial}{\partial y_n}\right)F(y_1, \ldots, y_n),\label{pp}
\end{equation}
where $h \in k[x_1, \ldots, x_n]$ and $F \in k[y_1, \ldots, y_n]$.  For a positive integer $j$, differentiation induces a perfect pairing
$$ 
k[x_1, \ldots, x_n]_j \times k[y_1, \ldots, y_n]_j \to k, \quad (h, F) \mapsto h \circ F;
$$
it is often referred to as the {\it polar pairing}.  For $F \in k[y_1, \ldots, y_n]_j$, we now introduce a homogeneous ideal, called the annihilator of $F$, as follows:
$$
F^{\perp} := \{h\in k[x_1, \ldots, x_n]\, \mid \, h \circ F = 0 \},
$$
which is clearly independent of scaling and thus is well-defined for $F$ in the projective space $\PP(k[y_1, \ldots, y_n]_j)$. It is well-known that the quotient $k[x_1, \ldots, x_n] / F^{\perp}$ is a standard graded local Artinian Gorenstein algebra of socle degree $j$ and the following holds (cf.~\cite[Lemma 2.12]{IK}):

\begin{proposition} \label{prop-correspondence}
The correspondence $F \mapsto k[x_1, \ldots, x_n]/F^{\perp}$ induces a bijection
$$
\PP(k[y_1, \ldots, y_n]_j)  \to
\left\{ 
 	\begin{array}{l} 
		\text{local Artinian Gorenstein algebras $k[x_1, \ldots, x_n]/I$}\\ 
		\text{of socle degree $j$, where the ideal $I$ is homogeneous}\\
		 \end{array} \right\}.
$$
\end{proposition}

\begin{remark}\label{invsysrem}
\noindent Given a homogenous ideal $I \subset k[x_1, \ldots, x_n]$ such that $k[x_1, \ldots, x_n]/I$ is a local Artinian Gorenstein algebra of socle degree $j$, Proposition \ref{prop-correspondence} implies that there is a  form $F \in  k[y_1, \ldots, y_n]_j$, unique up to scaling, such that $I = F^{\perp}$. In fact, the uniqueness part of this statement can be strengthened: if $I\subset F^{\perp}$, then $I = F^{\perp}$ and all forms with this property are mutually proportional. Indeed, $I\subset F^{\perp}$ implies $I_j\subset F^{\perp}$, where $I_j:=I\cap k[x_1, \ldots, x_n]_j$, and the claim follows from the fact that $I_j$ has codimension 1 in $k[x_1, \ldots, x_n]_j$. Any such form $F$ is called {\it a {\rm (}homogeneous{\rm )} Macaulay inverse system for $k[x_1, \ldots, x_n]/I$} and its image in $\PP(k[y_1, \ldots, y_n]_j)$ is called {\it the {\rm (}homogeneous{\rm )} Macaulay inverse system for $k[x_1, \ldots, x_n]/I$}.
\end{remark}

We have (see \cite[Proposition 2.11]{AI}):

\begin{prop} \label{P:inverse-system} \it
For any $\bbf \in (k[x_1, \ldots, x_n]_{d}^{\oplus n})_{\Res}$, 
the form $\bA(\bbf)$ is a Macau\-lay inverse system for the algebra $M(\bbf)$.
\end{prop}

\noindent By Proposition \ref{P:inverse-system}, the morphism $\bA$ can be thought of as a map assigning to every element $\bbf \in (k[x_1, \ldots, x_n]_{d}^{\oplus n})_{\Res}$ a particular Macaulay inverse system for the algebra $M(\bbf)$.

We now let $U_{\Res} \subset k[y_1, \ldots, y_n]_{n(d-1)}$ be the locus of forms $F$ such that the subspace $F^{\perp} \cap k[x_1, \ldots, x_n]_d$ is $n$-dimensional and has a basis with nonvanishing resultant. It is easy to see that $U_{\Res}$ is locally closed in $k[y_1, \ldots, y_n]_{n(d-1)}$, hence is a variety (see, e.g., Proposition \ref{onto1} below for details). By Proposition \ref{P:inverse-system}, the image of $\bA$ is contained in $U_{\Res}$. Moreover, if $F \in U_{\Res}$, then for the ideal $I\subset k[x_1, \ldots, x_n]$ generated by $F^{\perp} \cap k[x_1, \ldots, x_n]_{d}$, we have the inclusion $I \subset F^{\perp}$. By Remark \ref{invsysrem}, the form $F$ is the inverse system for $k[x_1, \ldots, x_n]/I$, and therefore $F=\bA(\bbf)$ for some basis $\bbf=(f_1,\dots,f_n)$ of $F^{\perp} \cap k[x_1, \ldots, x_n]_{d}$. Thus, we have proved:

\begin{prop} \label{P:image} \it $\im(\bA)=U_{\Res}$. 
\end{prop}

The constructions of the morphism $\bA$ can be projectivized. Indeed, denote by $\Gr(n, k[x_1, \ldots, x_n]_{d})$ the Grassmannian of $n$-dimensional subspaces of $k[x_1, \ldots, x_n]_{d}$. The resultant $\Res$ on $k[x_1, \ldots, x_n]_{d}^{\oplus n}$ descends to a section, also denoted by $\Res$, of a power of the very ample generator of the Picard group of $\Gr(n, k[x_1, \ldots, x_n]_{d})$.  Let $\Gr(n, k[x_1, \ldots, x_n]_{d})_{\Res}$ be the affine open subvariety where $\Res$ does not vanish; it consists of all $n$-dimensional subspaces of $k[x_1, \ldots, x_n]_{d}$ having a basis with nonzero resultant. Consider the morphism 
$$
(k[x_1, \ldots, x_n]_{d}^{\oplus n})_{\Res} \to \Gr(n, k[x_1, \ldots, x_n]_{d})_{\Res}, \quad  \bbf = (f_1, \ldots, f_n) \mapsto \langle f_1, \ldots, f_n \rangle,
$$
where $\langle\,\cdot\,\rangle$ denotes linear span. Then, by equivariance property (\ref{E:equiv2}), the morphism $\bA$ composed with the projection $k[y_1, \ldots, y_n]_{n(d-1)}\setminus\{0\}\to\PP(k[y_1, \ldots, y_n]_{n(d-1)})$ factors as
$$
(k[x_1, \ldots, x_n]_{d}^{\oplus n})_{\Res} \to \Gr(n, k[x_1, \ldots, x_n]_{d})_{\Res} \xrightarrow{\hat\bA} \PP(k[y_1, \ldots, y_n]_{n(d-1)}).
$$
By Proposition \ref{P:inverse-system}, the morphism $\hat\bA$ can be thought of as a map assigning to every subspace $W\in\Gr(n, k[x_1, \ldots, x_n]_{d})_{\Res}$ {\it the}\, Macaulay inverse system for the algebra $M(\bbf)$, where $\bbf=(f_1,\dots,f_n)$ is any basis of $W$.

By Proposition \ref{P:image} we have $\im(\hat\bA)=\PP(U_{\Res})$, where $\PP(U_{\Res})$ is the image of $U_{\Res}$ in the projective space $\PP(k[y_1, \ldots, y_n]_{n(d-1)})$. With a little extra effort one obtains (see \cite[Proposition 2.13]{AI}):

\begin{prop} \label{P:imagehat} \it The morphism $\hat\bA:\Gr(n, k[x_1, \ldots, x_n]_{d})_{\Res}\to\PP(U_{\Res})$ is an isomorphism.
\end{prop}

\noindent This result will be utilized in our considerations of the relevant catalecticant varieties in the next section.

\section{The catalecticant schemes and varieties}\label{schemes}
\setcounter{equation}{0}

Let
$$
K:=\dim_k k[x_1,\dots,x_n]_d={d+n-1 \choose n-1}.
$$
Consider the quasiaffine variety
$$
U:=U_{K-n}(n(d-1)-d, d;n)\subset k[y_1,\dots,y_n]_{n(d-1)}
$$
and the affine subvariety
$$
V:=V_{K-n}(n(d-1)-d, d;n)\subset k[y_1,\dots,y_n]_{n(d-1)}
$$
as defined in \cite[p.~5]{IK}. Specifically, set
$$
L:=\dim_{k} k[y_1,\dots,y_n]_{n(d-1)-d}={n(d-1)-d+n-1 \choose n-1}
$$
and let $\{{\mathtt m}_1,\dots,{\mathtt m}_K\}$, $\{{\mathbf m}_1,\dots,{\mathbf m}_L\}$ be the standard monomial bases in the spaces $k[x_1,\dots,x_n]_d$ and $k[y_1,\dots,y_n]_{n(d-1)-d}$, respectively, with the monomials numbered in accordance with some orders, which we will fix from now on. For a form $F\in k[y_1,\dots,y_n]_{n(d-1)}$ let $F_j:={\mathtt m}_j\circ F\in k[y_1,\dots,y_n]_{n(d-1)-d}$, $j=1,\dots,K$,\linebreak where $\circ$ is defined in (\ref{pp}). Expanding $F_1,\dots,F_K$ with respect to $\{{\mathbf m}_1,\dots,{\mathbf m}_L\}$, we obtain an $L\times K$-matrix $D(F)$ called the {\it catalecticant matrix}. Then the varieties $U$ and $V$ are described as
$$
\begin{array}{l}
U=\{F \in k[y_1,\dots,y_n]_{n(d-1)}\mid \rank D(F)=K-n\},\\
\vspace{-0.3cm}\\
V=\{F\in k[y_1,\dots,y_n]_{n(d-1)}\mid \rank D(F)\le K-n\}.
\end{array}
$$
Note that $U$ is a dense open subset of $V$ (see \cite[Lemma 3.5]{IK}).

Clearly, $V\subset k[y_1,\dots,y_n]_{n(d-1)}$ is the affine subvariety given by the condition of the vanishing of all $(K-n+1)$-minors of $D(F)$. Observe that for $n=2$ we have $K=d+1$, $L=d-1$, and therefore the matrix $D(F)$ has no $(K-1)$-minors, hence $V=k[y_1,y_2]_{2(d-1)}$. Similarly, for $n=3$, $d=2$, we have $K=6$, $L=3$, therefore $D(F)$ has no $(K-2)$-minors, hence $V=k[y_1,y_2,y_3]_3$. Notice that in all other cases $L\ge K$, and therefore $V$ is a proper affine subvariety of $k[y_1,\dots,y_n]_{n(d-1)}$ unless $n=2$ or $n=3$, $d=2$. 

We will also consider the corresponding affine scheme ${\mathbf V}$ defined as follows. Let $M:=\dim_k k[y_1,\dots,y_n]_{n(d-1)}$ and $\{{\rm m}_1,\dots,{\rm m}_M\}$ be the standard monomial basis in the space $k[y_1,\dots,y_n]_{n(d-1)}$ with the monomials numbered in accordance with some order. Let $k[z_1,\dots,z_M]$ be the ring of polynomials in the coefficients of the expansion of a form in $k[y_1,\dots,y_n]_{n(d-1)}$ with respect to $\{{\rm m}_1,\dots,{\rm m}_M\}$.

\begin{definition}\label{defv}
Denote by $J_1$ the ideal in $k[z_1,\dots,z_M]$ generated by all $(K-n+1)$-minors of the matrix $D(F)$. We set
$$
{\mathbf V}:=\hbox{Spec}\,k[z_1,\dots,z_M]/J_1.
$$
\end{definition}

\noindent Clearly, we have ${\mathbf V}_{\red}=V$. 

Next, let $T:=(t_0,t_1,\dots,t_{n(d-1)})=(1,n,\dots,n,1)$ be the Gorenstein sequence from Hilbert function (\ref{hf}), which is symmetric about $n(d-1)/2$. Consider the quasiaffine variety $\Gor(T)$ that consists of all forms $F\in k[y_1,\dots,y_n]_{n(d-1)}$ such that the Hilbert function of the standard graded local Artinian Gorenstein algebra $k[x_1,\dots,x_n]/F^{\perp}$ is $T$. Clearly, $\Gor(T)$ is an open subset of the affine subvariety $\Gor_{\le}(T)\subset k[y_1,\dots,y_n]_{n(d-1)}=k[z_1,\dots,z_M]$ consisting of all forms $F$ for which the Gorenstein sequence of $k[x_1,\dots,x_n]/F^{\perp}$ does not exceed $T$. Analogously to $V$, the variety $\Gor_{\le}(T)$ is defined by the vanishing of all $(t_i+1)$-minors of the corresponding matrices constructed analogously to $D(F)$, for $i=1,\dots,n(d-1)-1$. Let $J_2$ be the ideal in $k[z_1,\dots,z_M]$ generated by all such minors (see \cite[p.~8]{IK} for details) and consider the affine scheme $\bfGor_{\le}(T):=\hbox{Spec}\,k[z_1,\dots,z_M]/J_2$. Then $\bfGor_{\le}(T)_{\red}=\Gor_{\le}(T)$, and we introduce another main player of this article, the scheme $\bfGor(T)$, as follows.

\begin{definition}\label{defgor}
Set $\bfGor(T)$ to be the open subscheme of $\bfGor_{\le}(T)$ defined by the open subset $\Gor(T)$ of $\Gor_{\le}(T)$.
\end{definition}

\noindent Clearly, we have $\bfGor(T)_{\red}=\Gor(T)$. Following \cite{IK}, we call ${\mathbf V}$, $\bfGor(T)$ the {\it catalecticant schemes} and $V$, $\Gor(T)$ the {\it catalecticant varieties}.

\begin{remark}\label{moregeneralcatal}
Note that \cite{IK} introduces more general catalecticant varieties and schemes, but in this paper we only focus on $V$, $\Gor(T)$, ${\mathbf V}$, $\bfGor(T)$.\end{remark}

We have the obvious inclusions
\begin{equation}
U_{\Res}\subset\Gor(T)\subset U\subset V,\label{setincl}
\end{equation}
where $U_{\Res}\subset k[y_1,\dots,y_n]_{n(d-1)}$ was defined in Section \ref{morphism}. To better understand the relationship between $U_{\Res}$, $\Gor(T)$, $U$ and $V$, we will now introduce a certain closed subset of $U$. 

Cover $U$ by open subsets ${\rm U}_{\alpha}$, each of which is given by the condition of the nonvanishing of a particular $(K-n)$-minor of the catalecticant matrix $D(F)$. In what follows, on each ${\rm U}_{\alpha}$ we will define a regular function $R_{\alpha}$. Let, for instance, ${\rm U}_{\alpha_{{}_0}}$ be the subset of $U$ described by the nonvanishing of the principal $(K-n)$-minor of $D(F)$. For $F\in {\rm U}_{\alpha_{{}_0}}$ we will now find a canonical basis of the solution set ${\mathcal S}(F)$ of the homogeneous system
$D(F)\gamma=0$,
where $\gamma$ is a column-vector in $k^K$. Since $\rank D(F)=K-n$, one has $\dim_{k}{\mathcal S}(F)=n$. Split $D(F)$ into blocks as follows:
$$
D(F)=\left(
\begin{array}{c}
\boxed{A(F)} \quad \boxed{B(F)}\\
\vspace{-0.3cm}\\
\boxed{\hspace{0.7cm}C(F)\hspace{0.7cm}}
\end{array}
\right),
$$
where $A(F)$ has size $(K-n)\times(K-n)$ (recall that $\det A(F)\ne 0$), $B(F)$ has size $(K-n)\times n$, and $C(F)$ has size $(L-K+n)\times K$. We also split the column-vector $\gamma$ as $\gamma=\left(\begin{array}{c}\gamma'\\\gamma''\end{array}\right)$, where $\gamma$ is in $k^{K-n}$ and $\gamma''$ is in $k^n$. Then ${\mathcal S}(F)$ is given by the  condition
$
\gamma'=-A(F)^{-1}B(F)\gamma''.
$
Therefore, the vectors
$$
\gamma_j(F):=\left(\begin{array}{c}-A(F)^{-1}B(F){\mathbf e}_j\\{\mathbf e}_j\end{array}\right),\quad j=1,\dots,n,
$$
form a basis of ${\mathcal S}(F)$ for every $F\in {\rm U}_{\alpha_{{}_0}}$, where ${\mathbf e}_j$ is the $j$th standard basis vector in $k^n$. 

Clearly, the components $\gamma_j^1,\dots,\gamma_j^K$ of $\gamma_j$ are regular functions on ${\rm U}_{\alpha_{{}_0}}$ for each $j$, and we define
$r_{j,\alpha_{{}_0}}:=\sum_{i=1}^K \gamma_j^i\,{\mathtt m}_i$, $j=1,\dots,n$,
where, as before, $\{{\mathtt m}_1,\dots,{\mathtt m}_K\}$ is the standard monomial basis in $k[x_1,\dots,x_n]_d$. Then the $d$-forms $r_{1,{\alpha_{{}_0}}}(F),\dots,r_{n,{\alpha_{{}_0}}}(F)$ constitute a basis of the intersection $F^{\perp}\cap k[x_1,\dots,x_n]_d$ for every $F\in {\rm U}_{\alpha_{{}_0}}$.  Set
$
R_{\alpha_{{}_0}}:=\Res(r_{1,\alpha_{{}_0}},\dots,r_{n,\alpha_{{}_0}}).
$
Clearly, $R_{\alpha_{{}_0}}$ is a regular function on ${\rm U}_{\alpha_{{}_0}}$, and we define $Z_{\alpha_{{}_0}}$ to be its zero locus.

Arguing as above for every ${\rm U}_{\alpha}$, we introduce a regular function $R_{\alpha}$ on ${\rm U}_{\alpha}$ and its zero locus $Z_{\alpha}$. Notice that if for some $\alpha$, $\alpha'$ the intersection ${\rm U}_{\alpha,\alpha'}:={\rm U}_{\alpha}\cap {\rm U}_{\alpha'}$ is nonempty, then $Z_{\alpha}\cap {\rm U}_{\alpha,\alpha'} =Z_{\alpha'}\cap {\rm U}_{\alpha,\alpha'}$. Thus, the loci $Z_{\alpha}$ glue together into a closed subset $Z$ of $U$. If $U'$ is an irreducible component of $U$, then the intersection $Z\cap U'$ is either a hypersurface in $U'$, or all of $U'$, or empty. Notice also that $Z$ is $\GL_n$-invariant, which follows from the general formula
$$
(CF)^{\perp}\cap k[x_1,\dots,x_n]_j=C^{-t}\,(F^{\perp}\cap k[x_1,\dots,x_n]_j),\quad j=0,\dots,n(d-1),
$$
for all $C\in\GL_n$, $F\in k[y_1,\dots,y_n]_{n(d-1)}$.

We will now establish:

\begin{proposition}\label{onto1} One has $U_{\Res}=\Gor(T)\setminus Z=U\setminus Z=V\setminus\overline{Z}$.
\end{proposition}

\begin{proof} It is clear that $U_{\Res}=U\setminus Z$, thus inclusions (\ref{setincl}) imply $U_{\Res}=\Gor(T)\setminus Z=U\setminus Z$. Further, to see that $U\setminus Z=V\setminus\overline{Z}$, we need to prove that $V\setminus U\subset\overline{Z}$. As shown in the proof of \cite[Lemma 3.5]{IK}, in every neighborhood of every form $F\in V\setminus U$ there exists $\hat F\in U$ such that all elements of $\hat F^{\perp} \cap k[x_1,\dots,x_n]_d $ have a common zero away from the origin. Thus, $F\in\overline{Z}$ as required. \end{proof}

Next, recall that by Proposition \ref{P:imagehat} the morphism $\hat\bA: \Gr(n, k[x_1, \ldots, x_n]_{d})_{\Res}\to \PP(U_{\Res})$ is an isomorphism. Therefore, we have
$$
\dim_k \PP(U_{\Res})=\dim_k \Gr(n, k[x_1, \ldots, x_n]_{d})=Kn-n^2,
$$
which implies
\begin{equation}
\dim_k U_{\Res}=Kn-n^2+1=:N.\label{dim1}
\end{equation}
As $U_{\Res}$ is irreducible, we obtain the following result:

\begin{theorem}\label{main8} There exist irreducible components $\Gor(T)^{\circ}$, $U^{\circ}$, $V^{\circ}$ of the varieties $\Gor(T)$, $U$, $V$, respectively, such that
$
U_{\Res}=\Gor(T)^{\circ}\setminus Z=U^{\circ}\setminus Z=V^{\circ}\setminus\overline{Z},
$
with
$
\dim_k\Gor(T)^{\circ}=\dim_k U^{\circ}=\dim_k V^{\circ}=N,
$
where $N$ is defined in {\rm (\ref{dim1})}.
\end{theorem}

\begin{remark}\label{compar1} A fact similar to Theorem \ref{main8} was obtained in \cite{IK}. Specifically, Theorem 4.17 of \cite{IK} shows that $\Gor(T)$ has an irreducible component containing $U_{\Res}$ as a dense subset and the dimension of this component is equal to $N$. The proof given in \cite{IK} does not explicitly utilize the morphism $\bA$ and is somewhat brief overall. Also, Theorem 4.19 of \cite{IK} (cf.~Corollary 4.18 therein) yields that $U_{\Res}$ is dense in an irreducible component of $V$ in the following cases: (i) $n=3$, $d\ge 3$,\linebreak (ii) $n=4$, $d=2,3$, (iii) $n=5$, $d=2$. In comparison with these results, Theorem \ref{main8} stated above is more precise because:
\begin{itemize}

\item it treats both $\Gor(T)$ and $V$ simultaneously for all $n,d$;

\item it shows that $U_{\Res}$ is in fact open (not just dense) in an irreducible component of each of $\Gor(T)$ and $V$ and explicitly describes the closed complement to $U_{\Res}$ in terms of the subset $Z$;

\item its proof gives a complete argument for the formula for $\dim_k U_{\Res}$. 

\end{itemize}
\vspace{0.3cm}

In the next section we will see that the set $U_{\Res}$ lies in the smooth part of $V^{\circ}$, i.e., that all singularities of $V^{\circ}$ are contained in $\bar{Z}$. In fact, we will arrive at a stronger conclusion by investigating the smoothness of the schemes ${\mathbf V}$ and $\bfGor(T)$ at the points of $U_{\Res}$.

\end{remark}

\section{Smoothness of ${\mathbf V}$ and $\bfGor(T)$ along $U_{\Res}$}\label{results}
\setcounter{equation}{0}

In this section we prove:

\begin{theorem}\label{main888} Each of the catalecticant schemes ${\mathbf V}$ and $\bfGor(T)$ is smooth at every point lying in $U_{\Res}$.
\end{theorem}

\begin{remark}\label{smoothness}
Let $X$ be one of the schemes ${\mathbf V}$, $\bfGor(T)$. The smoothness of $X$ at a point $x\in X$ is understood as the fact that the ring ${\mathcal O}_{X,x}$ is a regular local ring. This is equivalent to the identity $\dim T_x(X)=\dim {\mathcal O}_{X,x}$, where the Zariski tangent space $T_x(X)$ to $X$ at $x$ is regarded as a vector space over the residue field at $x$ (see, e.g., \cite[p.~158]{GW}).
\end{remark}

\begin{remark}\label{compar2} The smoothness of $\bfGor(T)$ along $U_{\Res}$ follows from \cite[Theorem 1.4]{K} as the algebra $M(\bbf)$ is a complete intersection for every $\bbf\in(k[x_1, \ldots, x_n]_d^{\oplus n})_{\Res}$. At the same time, the smoothness of ${\mathbf V}$ along $U_{\Res}$ appears to be only known in cases (i)--(iii) specified in Remark \ref{compar1} (see \cite[Theorem 4.19]{IK}). The proof of Theorem \ref{main888} given below works simultaneously for both ${\mathbf V}$ and $\bfGor(T)$ with arbitrary $n,d$. It does not depend on results of \cite{K} and is much more elementary. Indeed, our argument does not rely either on Laudal's description of hulls or on results concerning the tangent and obstruction space for the scheme $GradAlg(H)$ (cf.~\cite[p.~610]{K}). Apart from the formulas for the dimensions of the tangent spaces to ${\mathbf V}$ and $\bfGor(T)$ found in \cite[Theorems 3.2, 3.9]{IK}, it only utilizes the Koszul resolution of the algebra $M(\bbf)$. Our approach is combinatorial, and, as part of the proof, we obtain a number of combinatorial identities that are of independent interest.
\end{remark}

\begin{proof}[Proof of Theorem {\rm \ref{main888}}] The theorem is obvious in the cases $n=2$ and $n=3$, $d=2$, thus everywhere below we assume that $n\ge 3$ and that $d\ge 3$ if $n=3$. We will compute the dimensions of the tangent spaces $T_F({\mathbf V})$ to ${\mathbf V}$ and $T_F(\bfGor(T))$ to $\bfGor(T)$ at every closed point $F\in U_{\Res}$. In view of Theorem \ref{main8}, we only need to show that
$\dim_{k} T_F({\mathbf V})=\dim_{k} T_F(\bfGor(T))=N$ for all closed points $F\in U_{\Res}$, where $N$ is the number introduced in (\ref{dim1}).

Fix a closed point $F\in U_{\Res}$, let $I:=F^{\perp}$ and set $I_j:=I\cap k[x_1, \ldots, x_n]_j$ for all $j\ge0$. By \cite[Theorem 3.2]{IK} we then see
\begin{equation}
\begin{array}{l}
\displaystyle\dim_{k} T_F({\mathbf V})=\dim_{k}k[y_1, \ldots, y_n]_{n(d-1)}-\dim_{k}  I_{d}\, I_{n(d-1)-d}=\\
\vspace{-0.1cm}\\
\displaystyle\hspace{4cm}{n(d-1)+n-1 \choose n-1}-\dim_{k}  I_{d}\, I_{n(d-1)-d}.
\end{array}\label{formfortang}
\end{equation}
In formula (\ref{formfortang}) and everywhere below, for two linear subspaces ${\mathcal U}\subset k[x_1, \ldots, x_n]_{i}$, ${\mathcal V}\subset k[x_1, \ldots, x_n]_{j}$, the product ${\mathcal U}{\mathcal V}$ denotes the linear subspace of $k[x_1, \ldots, x_n]_{i+j}$ spanned by all products $g h$, with $g\in {\mathcal U}$, $h\in {\mathcal V}$.

As $n(d-1)-2d\ge 0$, we have $I_{n(d-1)-d}=k[x_1, \ldots, x_n]_{n(d-1)-2d}\, I_{d}$, and therefore 
\begin{equation}
\dim_{k} T_F({\mathbf V})={n(d-1)+n-1 \choose n-1}-\dim_{k} (k[x_1, \ldots, x_n]_{n(d-1)-2d}\,I_{d}^2).\label{dimt1}
\end{equation}
Together with \cite[Theorem 3.9]{IK}, formula (\ref{dimt1}) yields the equality $\dim_{k}T_F({\mathbf V})=\dim_{k}T_F(\bfGor(T))$ (cf.~\cite[p.~115]{IK}). Hence, to prove the theorem for both ${\mathbf V}$ and $\bfGor(T)$, we only need to establish the identity
\begin{equation}
{n(d-1)+n-1 \choose n-1}-\dim_{k} (k[x_1, \ldots, x_n]_{n(d-1)-2d}\,I_{d}^2)=N.\label{desiredestim}
\end{equation}

In order to obtain (\ref{desiredestim}), one has to compute $\dim_{k} (k[x_1, \ldots, x_n]_{n(d-1)-2d}\,I_{d}^2)$. We fix a basis $\bbf=(f_1,\dots,f_n)$ in $I_d$ and start with the following simple observation:

\begin{lemma}\label{independenceofpartials}\it For any integer $\ell\ge 1$ the products of forms $f_{i_1}\cdots f_{i_{\ell}}$,\linebreak with $1\le i_1\le\dots\le i_{\ell}\le n$, are linearly independent in $k[x_1, \ldots, x_n]_{\ell d}$. In particular, one has
$\dim_{k}I_{d}^2=n(n+1)/2.$
\end{lemma}

\begin{proof} Suppose that there is a zero linear combination 
$$
\sum_{1\le i_1\le\dots\le i_{\ell}\le n}\alpha_{i_1,\dots,i_{\ell}}f_{i_1}\cdots f_{i_{\ell}}=0.
$$ 
This means that the morphism $\bbf=(f_1,\dots,f_n)\co {\mathbb A}^n(k)\to {\mathbb A}^n(k)$ sends all of ${\mathbb A}^n(k)$ into the hypersurface given by the equation
$
\sum_{1\le i_1\le\dots\le i_{\ell}\le n}\alpha_{i_1,\dots,i_{\ell}}x_{i_1}\cdots x_{i_{\ell}}=0,
$
which by Lemma \ref{fourconds} contradicts the condition $\Res(\bbf)\ne 0$. \end{proof}

By Lemma \ref{independenceofpartials}, we have
\begin{equation}
\makebox[250pt]{$\begin{array}{l}
\displaystyle\dim_{k} (k[x_1, \ldots, x_n]_{n(d-1)-2d}\,I_{d}^2)=\\
\vspace{-0.1cm}\\
\displaystyle\hspace{4.3cm}
\frac{n(n+1)}{2}\dim_k k[x_1, \ldots, x_n]_{n(d-1)-2d}-\dim_{k}{\mathcal R}=\\
\vspace{-0.1cm}\\
\displaystyle\hspace{4.3cm}
\frac{n(n+1)}{2}{n(d-1)-2d+n-1 \choose n-1}-\dim_{k}{\mathcal R},
\end{array}$}\label{dimviaR}
\end{equation}
where ${\mathcal R}$ is the vector space of linear relations among $e_{\ell}f_if_j$ in $k[x_1, \ldots, x_n]_{n(d-1)}$, with $i\le j$ and $\{e_{\ell}\}$ being any basis in $k[x_1, \ldots, x_n]_{n(d-1)-2d}$. Every relation in ${\mathcal R}$ has the form
\begin{equation}
\sum_{1\le i\le j\le n}\alpha_{ij}f_if_j=0,\label{generlinrel}
\end{equation}
with $\alpha_{ij}\in k[x_1, \ldots, x_n]_{n(d-1)-2d}$. In the following proposition we determine $\dim_{k}{\mathcal R}$.

\begin{proposition}\label{finaldim} One has
\begin{equation}
\displaystyle\dim_{k} {\mathcal R}=\sum_{m=3}^{\left[\frac{n(d-1)}{d}\right]}(-1)^{m-1}(m-1){n+1 \choose m}{n(d-1)-md+n-1 \choose n-1},\label{formdimmesR}
\end{equation}
where $[x]$ denotes the largest integer that is less than or equal to $x$.
\end{proposition}

\begin{proof} Let $\SSMat^{\ell}({\mathcal V})$ denote the vector space of skew-symmetric\linebreak $\ell \times \ell$-matrices with entries in a vector space ${\mathcal V}$. Clearly, $\SSMat^{\ell}({\mathcal V})$ is isomorphic to ${\mathcal V}^{\oplus\frac{\ell(\ell-1)}{2}}$. We need the following lemma: 

\begin{lemma}\label{lineardepens1}\it Consider the linear map
$$
\phi_{\rho}\co: k[x_1, \ldots, x_n]_{\rho}^{\oplus n} \to I_{\rho+d},\quad (h_1,\dots,h_n)\mapsto h_1 f_1+\cdots+h_n f_n,
$$
where $\rho\le n(d-1)-d$. Then
$
\ker(\phi_{\rho})=\left\{\bbf D\mid D\in \SSMat^n(k[x_1, \ldots, x_n]_{\rho-d})\right\}.
$
\end{lemma}

\begin{proof} The graded minimal free resolution of $I$ is known to coincide with the Koszul resolution
$
0\to {\mathcal F}_{n-1}\buildrel{d_{n-1}}\over{\to}\cdots\buildrel{d_2}\over{\to} {\mathcal F}_1\buildrel{d_1}\over{\to} {\mathcal F}_0\buildrel{d_0}\over{\to} I\to 0
$
(see, e.g., \cite[Theorem 1.1]{GVT}). Here one has
${\mathcal F}_0=k[x_1,\dots,x_n](-d)^{\oplus n}$, ${\mathcal F}_1=k[x_1,\dots,x_n](-2d)^{\oplus\frac{n(n-1)}{2}}$,
and we identify ${\mathcal F}_1$ with with $\SSMat^{n}(k[x_1,\dots,x_n](-2d))$. Upon this identification, the exactness condition $\im(d_1)=\ker(d_0)$ is precisely the assertion of the lemma.
\end{proof}

We will now continue the proof of the proposition. Write any relation in ${\mathcal R}$ as
$
\sum_{i=1}^n\left(\sum_{j=i}^n\alpha_{ij}f_j\right)f_i=0
$
(see (\ref{generlinrel})). By Lemma \ref{lineardepens1} with $\rho=n(d-1)-d$, there exists $D\in\SSMat^n(k[x_1,\dots,x_n]_{n(d-1)-2d})$ such that
$
{\bbf }\alpha={\bbf }D ,
$
where $\alpha$ is the lower-triangular $n\times n$-matrix whose entries are $\alpha_i^j:=\alpha_{ij}$, $i\le j$, with the upper and lower indices indicating the row and column numbers, respectively. Applying Lemma \ref{lineardepens1} to each column of the matrix $\alpha-D$ with $\rho=n(d-1)-2d$, we then see that there exist $D^j\in\SSMat^n(k[x_1,\dots,x_n]_{n(d-1)-3d})$, $j=1,\dots,n$, satisfying
\begin{equation}
\begin{array}{ll}
\displaystyle D_i^j=\sum_{{\ell}=1}^n(D^j)^{\ell}_if_{\ell},&  1\le i< j\le n,\\
\vspace{-0.5cm}\\
\displaystyle \alpha_{ij}-D_i^j=\sum_{{\ell}=1}^n(D^i)^{\ell}_jf_{\ell}, & 1\le i< j\le n,\\  
\vspace{-0.5cm}\\
\displaystyle \alpha_{ii}=\sum_{{\ell}=1}^n(D^i)^{\ell}_if_{\ell}, & i=1,\dots, n.\\
\end{array}\label{threesetsofeqs}
\end{equation}
The first two sets of equations in (\ref{threesetsofeqs}) imply
$\alpha_{ij}=\sum_{{\ell}=1}^n((D^i)^{\ell}_j+(D^j)^{\ell}_i)f_{\ell}$,\linebreak $1\le i< j\le n$, which together with the third set of equations shows that ${\mathcal R}$ is parametrized by the space $(\SSMat^n(k[x_1,\dots,x_n]_{n(d-1)-3d}))^{\oplus n}$. 

Let $q:=\left[\frac{n(d-1)}{d}\right]-1$. If $q=1$, the space ${\mathcal R}$ is trivial, which agrees with (\ref{formdimmesR}). Suppose that $q\ge 2$. Then the kernel of the parametrization of ${\mathcal R}$ consists of all $n$-tuples of matrices $(D^1,\dots,D^n)\in(\SSMat^n(k[x_1,\dots,x_n]_{n(d-1)-3d}))^{\oplus n}$ satisfying
\begin{equation}
\sum_{\ell=1}^n((D^i)^{\ell}_j+(D^j)^{\ell}_i)f_{\ell}=0, \quad 1\le i\le j\le n.\label{defw1alternative}
\end{equation}
By the skew-symmetricity of $D^i$ and Lemma \ref{lineardepens1}, identity (\ref{defw1alternative}) holds if and only if
\begin{equation}
R^j(D^i)+R^i(D^j)={\bbf}D^{ij}, \quad 1\le i\le j\le n,\label{cond1}
\end{equation}
for some $D^{ij}\in\SSMat^n(k[x_1,\dots,x_n]_{n(d-1)-4d})$, where for any matrix $M$ we denote by $R^i(M)$ its $i$th row. In fact, condition (\ref{cond1}) completely characterizes the vector space of all $n$-tuples $(D^1,\dots,D^n)\in(\SSMat^n(k[x_1,\dots,x_n]_{n(d-1)-3d}))^{\oplus n}$ that can occur in the right-hand side of the identities $R^i(D)={\bbf }D^i$, $i=1,\dots,n,$ for some $D\in \SSMat^n(k[x_1,\dots,x_n]_{n(d-1)-2d})$. We denote this vector space by ${\mathcal W}_1$.

Thus, we have
\begin{equation}
\begin{array}{l}
\displaystyle\dim_{k}{\mathcal R}=\frac{n^2(n-1)}{2}\dim_{k}k[x_1,\dots,x_n]_{n(d-1)-3d}-\dim_{k}{\mathcal W}_1=\\
\vspace{-0.3cm}\\
\displaystyle\hspace{3.5cm}\frac{n^2(n-1)}{2}{n(d-1)-3d+n-1 \choose n-1}-\dim_{k}{\mathcal W}_1,
\end{array}\label{dimRprefinal}
\end{equation}
and we will now find $\dim_{k}{\mathcal W}_1$. For any $1\le s\le q-1$ introduce the vector space ${\mathcal W}_s$ of ${s+n-1 \choose n-1}$-tuples of matrices $D^{i_1\cdots i_s}\in\SSMat^n(k[x_1,\dots,x_n]_{n(d-1)-(s+2)d})$, with $1\le i_1\le\dots\le i_s\le n$, satisfying
$$
\begin{array}{l}
R^p(D^{i_1\cdots i_s})+R^{i_1}(D^{p\, i_2\cdots i_s})+\dots+R^{i_s}(D^{p\, i_1\cdots i_{s-1}})={\bbf }D^{p\, i_1\dots i_s},\end{array}
$$
for all indices $1\le p\le i_1\le\dots\le i_s\le n$ and some skew-symmetric matrices $D^{p\, i_1\dots i_s}\in\SSMat^n(k[x_1,\dots,x_n]_{n(d-1)-(s+3)d})$. This is exactly the vector space of ${s+n-1 \choose n-1}$-tuples of matrices that can occur in the right-hand side of the identities
$$
\makebox[250pt]{$\begin{array}{l}
R^t(D^{j_1\cdots j_{s-1}})+R^{j_1}(D^{t\, j_2\cdots j_{s-1}})+\dots+R^{j_{s-1}}(D^{t\, j_1\cdots j_{s-2}})={\bbf} D^{t\, j_1\dots j_{s-1}},
\end{array}$}\label{aux7}
$$
for all indices $1\le t\le j_1\le\dots\le j_{s-1}\le n$ and some skew-symmetric matrices $D^{\ell_1\dots \ell_{s-1}}\in\SSMat^n(k[x_1,\dots,x_n]_{n(d-1)-(s+1)d})$. Let ${\mathcal K}_s$ be the subspace of ${\mathcal W}_s$ defined by the condition ${\bbf}D^{i_1\cdots i_s}=0$ for all $1\le i_1\le\dots\le i_s\le n$. Notice that ${\mathcal K}_{q-1}=0$ by Lemma \ref{lineardepens1}.

Also, for any $1\le s\le q-1$ consider the vector space of ${s+n-1 \choose n-1}$-tuples of matrices $A^{i_1\cdots i_s}\in\SSMat^n(k)$, with $1\le i_1\le\dots\le i_s\le n$, such that
$$
\makebox[250pt]{$\begin{array}{l}
R^p(A^{i_1\cdots i_s})+R^{i_1}(A^{p\, i_2\cdots i_s})+\dots+R^{i_s}(A^{p\, i_1\cdots i_{s-1}})=0,\\
\end{array}$}
$$
for all $1\le p\le i_1\le\dots\le i_s\le n$. The dimension $\delta_s$ of this space is not hard to find:
\begin{equation}
\delta_s=\frac{s(s+1)}{2}{n+s-1 \choose s+2}, \label{deltas}
\end{equation}
and we have   
\begin{equation}
\begin{array}{l}
\dim_{k}{\mathcal W}_{q-1}=\delta_{q-1}\dim_{k}k[x_1,\dots,x_n]_{n(d-1)-(q+1)d},\\
\vspace{-0.1cm}\\
\dim_{k}{\mathcal W}_{s}=\delta_s\dim_{k}k[x_1,\dots,x_n]_{n(d-1)-(s+2)d}+\dim_{k}{\mathcal W}_{s+1}-,\\
\vspace{-0.3cm}\\
\displaystyle\hspace{6cm}\dim_{k}{\mathcal K}_{s+1},\quad s=1,\dots, q-2.
\end{array}\label{dimkercalw}
\end{equation}

Formula (\ref{dimkercalw}) yields
\begin{equation}
\dim_{k}{\mathcal W}_{1}=\sum_{s=1}^{q-1}\delta_s\dim_{k}k[x_1,\dots,x_n]_{n(d-1)-(s+2)d}-\sum_{r=2}^{q-2}\dim_{k}{\mathcal K}_r.\label{dimk01}
\end{equation}

Further, if $q\ge 4$ one can compute $\dim_{k}{\mathcal K}_r$ for every $r=2,\dots,q-2$ by utilizing vector spaces similar to ${\mathcal W}_1,\dots,{\mathcal W}_{q-1}$ as follows. For any pair of integers $s\ge 1$, $r\ge 2$, with $s+r\le q-1$, we introduce the vector space ${\mathcal W}_{s;r}$ of ${s+n-1 \choose n-1}\times{r+n-1 \choose n-1}$-tuples of matrices $D^{i_1\cdots i_s;\iota_1\dots\iota_r}$ in $\SSMat^n(k[x_1,\dots,x_n]_{n(d-1)-(s+r+2)d})$, $1\le i_1\le\dots\le i_s\le n$, $1\le \iota_1\le\dots\le \iota_r\le n$, satisfying
$$
\begin{array}{l}
\hspace{-0.2cm}R^p(D^{i_1\cdots i_s;\iota_1,\dots\iota_r})+R^{i_1}(D^{p\, i_2\cdots i_s;\iota_1\dots\iota_r})+\hspace{-0.08cm}\cdots\hspace{-0.08cm}+R^{i_s}(D^{p\, i_1\cdots i_{s-1};\iota_1\dots\iota_r})\hspace{-0.1cm}=\hspace{-0.1cm}{\bbf}D^{p\, i_1\dots i_s;\iota_1\dots\iota_r},
\end{array}
$$
for all indices $1\le p\le i_1\le\dots\le i_s\le n$, $1\le \iota_1\le\dots\le \iota_r\le n$ and some $D^{p\, i_1\dots i_s;\iota_1\dots\iota_r}\in\SSMat^n(k[x_1,\dots,x_n]_{n(d-1)-(s+r+3)d})$. This is exactly the vector space of ${s+n-1 \choose n-1}\times{r+n-1 \choose n-1}$-tuples of matrices that can occur in the right-hand side of the identities
$$
\begin{array}{l}
R^t(D^{j_1\cdots j_{s-1};\iota_1\dots\iota_r})+R^{j_1}(D^{t\, j_2\cdots j_{s-1};\iota_1\dots\iota_r})+\dots+\\
\vspace{-0.4cm}\\
\hspace{6cm}R^{j_{s-1}}(D^{t\, j_1\cdots j_{s-2};\iota_1\dots\iota_r})={\bbf}D^{t\, j_1\dots j_{s-1};\iota_1\dots\iota_r},\end{array}
$$
for all indices $1\le t\le j_1\le\dots\le j_{s-1}\le n$, $1\le \iota_1\le\dots\le \iota_r\le n$ and some $D^{\ell_1\dots \ell_{s-1};\iota_1\dots\iota_r}\in\SSMat^n(k[x_1,\dots,x_n]_{n(d-1)-(s+r+1)d})$. Further, let ${\mathcal K}_{s;r}$ be the subspace of ${\mathcal W}_{s;r}$ defined by the condition
${\bbf}D^{i_1\cdots i_s;\iota_1\dots\iota_r}=0$ for all $1\le i_1\le\dots\le i_s\le n$, $1\le \iota_1\le\dots\le \iota_r\le n$.
Notice that by Lemma \ref{lineardepens1} one has ${\mathcal K}_{s,r}=0$ if $s+r=q-1$.

For every $2\le r\le q-2$ we have
\vspace{-0.05cm}
\begin{equation}
\dim_{k}{\mathcal W}_{q-r-1;r}=\delta_{q-r-1}{n+r-1\choose n-1}\dim_{k}k[x_1,\dots,x_n]_{n(d-1)-(q+1)d},\label{dimwrs}
\end{equation}
and for every $s=1,\dots, q-r-2$ we have
\vspace{-0.05cm}
\begin{equation}
\makebox[250pt]{$\begin{array}{l}
\displaystyle\hspace{1.5cm}\dim_{k}{\mathcal W}_{s;r}=\displaystyle\delta_s{n+r-1\choose n-1}\dim_{k}k[x_1,\dots,x_n]_{n(d-1)-(s+r+2)d}+\\
\vspace{-0.3cm}\\
\hspace{8.4cm}\displaystyle\dim_{k}{\mathcal W}_{s+1;r}-\dim_{k}{\mathcal K}_{s+1;r}.
\end{array}$}\label{dimrrr}
\end{equation}
Now, observe that ${\mathcal K}_r$, $r=2,\dots,q-2$ is parametrized by ${\mathcal W}_{1;r}$, with ${\mathcal K}_{1;r}$ being the kernel of this parametrization. Then (\ref{dimk01}), (\ref{dimwrs}), (\ref{dimrrr}) yield 
$$
\begin{array}{l}
\displaystyle\dim_{k}{\mathcal W}_1=\sum_{s=1}^{q-1}\delta_s\dim_{k}k[x_1,\dots,x_n]_{n(d-1)-(s+2)d}-\\
\vspace{-0.3cm}\\
\displaystyle\hspace{2.5cm}\sum_{r=2}^{q-2}{n+r-1\choose n-1}\sum_{s=1}^{q-r-1}\delta_s\dim_{k}k[x_1,\dots,x_n]_{n(d-1)-(s+r+2)d}+\\
\vspace{-0.5cm}\\
\displaystyle\hspace{8.4cm}\sum_{s\ge 1,r\ge 2,\, s+r\le q-2}\dim_{k}{\mathcal K}_{s;r}. 
\end{array}
$$

In order to determine $\dim_{k}{\mathcal K}_{s;r}$, with $s\ge 1$, $r\ge 2$, $s+r\le q-2$, one introduces further vector spaces analogous to ${\mathcal W}_s$, ${\mathcal W}_{s;r}$ by adding more indices to skew-symmetric matrices with elements in an appropriate space. Continuing this process, we arrive at the following formula:
\begin{equation}
\makebox[250pt]{$\begin{array}{l}
\displaystyle\dim_{k}{\mathcal W}_1=\sum_{s=1}^{q-1}\delta_s\dim_{k}k[x_1,\dots,x_n]_{n(d-1)-(s+2)d}+\\
\vspace{-0.6cm}\\
\displaystyle\hspace{3cm}\sum_{\ell=1}^{q-3}(-1)^{\ell}\hspace{-0.7cm}\sum_{\scalebox{0.5}{$\begin{array}{c}
r_1+\dots+r_{\ell}+s\le q-1,\\
r_1,\dots,r_{\ell-1}\ge 1,r_{\ell}\ge 2,s\ge 1
\end{array}$}}\hspace{-0.5cm}{n+r_1-1\choose n-1}\cdots{n+r_{\ell}-1\choose n-1}\times\\
\vspace{-0.3cm}\\
\displaystyle\hspace{5cm}\delta_s\dim_{k}k[x_1,\dots,x_n]_{n(d-1)-(s+r_1+\dots+r_{\ell}+2)d}=\\
\vspace{-0.3cm}\\
\displaystyle\hspace{0.5cm}\sum_{m=3}^{q+1}\left[\delta_{m-2}+\sum_{\ell=1}^{m-4}(-1)^{\ell}\hspace{-0.7cm}\sum_{\scalebox{0.5}{$\begin{array}{c}r_1+\dots+r_{\ell}+s=m-2,\\
r_1,\dots,r_{\ell-1}\ge 1,r_{\ell}\ge 2,s\ge 1
\end{array}$}}{n+r_1-1\choose n-1}\cdots {n+r_{\ell}-1\choose n-1}\delta_s\right]\times\\
\vspace{-0.5cm}\\
\hspace{8cm}\dim_{k}k[x_1,\dots,x_n]_{n(d-1)-md}.
\end{array}$}\label{dimkw1prefinal}
\end{equation}   
From (\ref{dimRprefinal}), (\ref{deltas}), (\ref{dimkw1prefinal}) we obtain
\begin{equation}
\hspace{-0.35cm}\makebox[250pt]{$\begin{array}{l}
\displaystyle\dim_{k}{\mathcal R}=2{n+1\choose 3}\dim_{k}k[x_1,\dots,x_n]_{n(d-1)-3d}-\\
\vspace{-0.3cm}\\
\hspace{6cm}\displaystyle3{n+1\choose 4}\dim_{k}k[x_1,\dots,x_n]_{n(d-1)-4d}-\\
\vspace{-0.3cm}\\
\hspace{1cm}\displaystyle\sum_{m=5}^{q+1}\left[\delta_{m-2}+\sum_{\ell=1}^{m-4}(-1)^{\ell}\hspace{-0.7cm}\sum_{\scalebox{0.5}{$\begin{array}{c}r_1+\dots+r_{\ell}+s=m-2,\\
r_1,\dots,r_{\ell}\ge 1,r_{\ell}\ge 2,s\ge 1
\end{array}$}}{n+r_1-1\choose n-1}\cdots{n+r_{\ell}-1\choose n-1}\delta_s\right]\times\\
\vspace{-0.4cm}\\
\hspace{8cm}\dim_{k}k[x_1,\dots,x_n]_{n(d-1)-md}.
\end{array}$}\label{dimrpprefinal}
\end{equation}
For $q=2,3$ formula (\ref{dimrpprefinal}) is easily seen to agree with (\ref{formdimmesR}) as required. For $q\ge 4$, the following identity will be established in the appendix below:
\begin{equation}
\makebox[250pt]{$\begin{array}{l}
\displaystyle\delta_{m-2}+\sum_{\ell=1}^{m-4}(-1)^{\ell}\hspace{-0.7cm}\sum_{\scalebox{0.5}{$\begin{array}{c}r_1+\dots+r_{\ell}+s=m-2,\\
r_1,\dots,r_{\ell-1}\ge 1,r_{\ell}\ge 2,s\ge 1
\end{array}$}}{n+r_1-1\choose n-1}\cdots{n+r_{\ell}-1\choose n-1}\delta_s=\\
\vspace{-0.4cm}\\
\displaystyle\hspace{4cm}(-1)^m(m-1){n+1\choose m}, \quad m=5,\dots,q+1.
\end{array}$}\label{identitytosolve3}
\end{equation}
Formulas (\ref{dimrpprefinal}) and (\ref{identitytosolve3}) clearly imply (\ref{formdimmesR}). The proof of the proposition is complete.\end{proof}

We will now finalize the proof of the theorem. By formulas (\ref{dimviaR}), (\ref{formdimmesR}), the left-hand side of (\ref{desiredestim}) is equal to
\begin{equation}
\begin{array}{l}
\displaystyle\sum_{m=0}^{\left[\frac{n(d-1)}{d}\right]}(-1)^{m-1}(m-1){n+1 \choose m}{nd-md-1 \choose n-1}.
\end{array}\label{dimt2}
\end{equation}
The fact that expression (\ref{dimt2}) coincides with the integer $N$ defined in (\ref{dim1}) is a consequence of the identities
\begin{equation}
\begin{array}{l}
\displaystyle\sum_{m=0}^{\left[\frac{n(d-1)}{d}\right]}(-1)^{m}{n+1 \choose m}{nd-md \choose n}={d+n-1 \choose n}
\end{array}\label{provenid1}
\end{equation}
and
\begin{equation}
\displaystyle\sum_{m=0}^{\left[\frac{n(d-1)}{d}\right]}(-1)^{m}{n \choose m}{nd-md-1 \choose n-1}=1,\label{provenid2}
\end{equation}
which will be verified in the appendix below. 

Indeed, using (\ref{provenid1}), (\ref{provenid2}) we have
$$
\begin{array}{l}
\displaystyle N= n {d+n-1\choose n-1}-n^2+1=\frac{n^2}{d} {d+n-1\choose n}-n^2+1=\\
\vspace{-0.3cm}\\
\displaystyle\frac{n^2}{d}\sum_{m=0}^{\left[\frac{n(d-1)}{d}\right]}(-1)^{m}{n+1 \choose m}{nd-md \choose n}-n^2+1=\\
\vspace{-0.3cm}\\
\displaystyle n \sum_{m=0}^{\left[\frac{n(d-1)}{d}\right]}(-1)^{m}(n-m){n+1 \choose m}{nd-md-1 \choose n-1}-\\
\vspace{-0.3cm}\\
\displaystyle (n^2-1)\sum_{m=0}^{\left[\frac{n(d-1)}{d}\right]}(-1)^{m}{n \choose m}{nd-md-1 \choose n-1}=\\
\vspace{-0.3cm}\\
\displaystyle n \sum_{m=0}^{\left[\frac{n(d-1)}{d}\right]}(-1)^{m}(n-m){n+1 \choose m}{nd-md-1 \choose n-1}-\\
\vspace{-0.3cm}\\
\displaystyle (n-1)\sum_{m=0}^{\left[\frac{n(d-1)}{d}\right]}(-1)^{m}(n-m+1){n+1 \choose m}{nd-md-1 \choose n-1}=\\
\vspace{-0.3cm}\\
\displaystyle\sum_{m=0}^{\left[\frac{n(d-1)}{d}\right]}(-1)^{m-1}(m-1){n+1 \choose m}{nd-md-1 \choose n-1},
\end{array}
$$
and this is indeed the expression shown in (\ref{dimt2}). The proof is complete. \end{proof}

\begin{remark}\label{alternpr}
As one of the referees pointed out, one can also obtain identity (\ref{desiredestim}) by a noncombinatorial argument based on standard facts such as \cite[Theorem 1.1.8]{BH}. We further note that one can give another noncombinatorial proof of (\ref{desiredestim}) using \cite[Remark 3.10]{IK}.
\end{remark}

\renewcommand{\theequation}{A.\arabic{equation}} 
\renewcommand{\thesubsection}{A.\arabic{subsection}}

\section*{Appendix: Some combinatorial identities}\label{combinatorialidentities}
\setcounter{equation}{0}

Our arguments in Section \ref{results} depend on certain combinatorial identities, namely, (\ref{identitytosolve3}), (\ref{provenid1}), (\ref{provenid2}). These identities are established below. We acknowledge that the proof for (\ref{provenid2}) has been communicated to us by Marko Riedel who kindly gave us permission to use it in our work. We start by reproducing his argument.

\subsection{Identity (\ref{provenid2})} We will show that
\begin{equation}
\displaystyle\sum_{m=0}^{\left[\frac{p(r-1)}{r}\right]}(-1)^{m}{p \choose m}{pr-mr-1 \choose p-1}=1\label{proved1}
\end{equation}
holds for all integers $p,r\ge 1$. Indeed, for any $0<\varepsilon<1$ write
$$
{pr-mr-1 \choose p-1}=\frac{1}{2\pi i}\int_{|z|=\varepsilon}\frac{(1+z)^{(p-m)r-1}}{z^p}dz,\quad \hbox{$m=0,\dots, \displaystyle\left[\frac{p(r-1)}{r}\right]$},
$$
which yields
$$
\begin{array}{l}
\displaystyle\sum_{m=0}^{\left[\frac{p(r-1)}{r}\right]}(-1)^{m}{p \choose m}{pr-mr-1 \choose p-1}=\\
\vspace{-0.3cm}\\
\displaystyle\sum_{m=0}^{\left[\frac{p(r-1)}{r}\right]}(-1)^{m}{p \choose m}\frac{1}{2\pi i}\int_{|z|=\varepsilon}\frac{(1+z)^{(p-m)r-1}}{z^p}dz=\\
\vspace{-0.3cm}\\
\displaystyle\frac{1}{2\pi i}\int_{|z|=\varepsilon}\left[\frac{(1+z)^{pr-1}}{z^p}\sum_{m=0}^{\left[\frac{p(r-1)}{r}\right]}(-1)^{m}{p \choose m}\frac{1}{(1+z)^{mr}}\right]dz=\\
\vspace{-0.3cm}\\
\displaystyle\frac{1}{2\pi i}\int_{|z|=\varepsilon}\left[\frac{(1+z)^{pr-1}}{z^p}\sum_{m=0}^p(-1)^{m}{p \choose m}\frac{1}{(1+z)^{mr}}\right]dz-\\
\vspace{-0.3cm}\\
\displaystyle\frac{1}{2\pi i}\int_{|z|=\varepsilon}\left[\frac{(1+z)^{pr-1}}{z^p}(-1)^p\frac{1}{(1+z)^{pr}}\right]dz=\\
\vspace{-0.3cm}\\
\displaystyle\frac{1}{2\pi i}\int_{|z|=\varepsilon}\left[\frac{(1+z)^{pr-1}}{z^p}\left(1-\frac{1}{(1+z)^r}\right)^p\right]dz+1=\\
\vspace{-0.3cm}\\
\displaystyle\frac{1}{2\pi i}\int_{|z|=\varepsilon}\frac{((1+z)^r-1)^p}{(1+z)z^p}dz+1=1
\end{array}
$$
as required. In the above calculation we utilized the obvious facts that
$$
\displaystyle\int_{|z|=\varepsilon}\left[\frac{(1+z)^{pr-1}}{z^p}\,\frac{1}{(1+z)^{mr}}\right]dz=0 \quad \hbox{for $\displaystyle\left[\frac{p(r-1)}{r}\right]<m<p$}
$$
and
$$
\displaystyle\frac{1}{2\pi i}\int_{|z|=\varepsilon}\left[\frac{(1+z)^{pr-1}}{z^p}\,\frac{1}{(1+z)^{pr}}\right]dz=(-1)^{p-1}.
$$

Note that for $r=2$  identity (\ref{proved1}) appears in \cite[formula (3.111)]{G} and that for this special case it was stated by B.~C.~Wong in {\it Amer.~Math. Monthly} in May 1930 as an open question.

\subsection{Identity (\ref{provenid1})} We will show that
\begin{equation}
\displaystyle\sum_{m=0}^{\left[\frac{p(r-1)}{r}\right]}(-1)^{m}{p+1 \choose m}{pr-mr \choose p}={p+r-1 \choose p}\label{proved2}
\end{equation}
holds for all integers $p,r\ge 1$. Our argument is similar to that used for obtaining (\ref{proved1}). Indeed, for any $0<\varepsilon<1$ write
$$
{pr-mr \choose p}=\frac{1}{2\pi i}\int_{|z|=\varepsilon}\frac{(1+z)^{(p-m)r}}{z^{p+1}}dz,\quad \hbox{$m=0,\dots, \displaystyle\left[\frac{p(r-1)}{r}\right]$},
$$
which yields
$$
\begin{array}{l}
\displaystyle\sum_{m=0}^{\left[\frac{p(r-1)}{r}\right]}(-1)^{m}{p+1 \choose m}{pr-mr \choose n}=\\
\vspace{-0.3cm}\\
\displaystyle\sum_{m=0}^{\left[\frac{p(r-1)}{r}\right]}(-1)^{m}{p+1 \choose m}\frac{1}{2\pi i}\int_{|z|=\varepsilon}\frac{(1+z)^{(p-m)r}}{z^{p+1}}dz=\\
\vspace{-0.3cm}\\
\displaystyle\frac{1}{2\pi i}\int_{|z|=\varepsilon}\left[\frac{(1+z)^{pr}}{z^{p+1}}\sum_{m=0}^{\left[\frac{p(r-1)}{r}\right]}(-1)^{m}{p+1 \choose m}\frac{1}{(1+z)^{mr}}\right]dz=\\
\vspace{-0.3cm}\\
\displaystyle\frac{1}{2\pi i}\int_{|z|=\varepsilon}\left[\frac{(1+z)^{pr}}{z^{p+1}}\sum_{m=0}^{p+1}(-1)^{m}{p+1 \choose m}\frac{1}{(1+z)^{mr}}\right]dz-\\
\vspace{-0.3cm}\\
\displaystyle\frac{1}{2\pi i}\int_{|z|=\varepsilon}\left[\frac{(1+z)^{pr}}{z^{p+1}}(-1)^{p+1}\frac{1}{(1+z)^{(p+1)r}}\right]dz=\\
\vspace{-0.3cm}\\
\displaystyle\frac{1}{2\pi i}\int_{|z|=\varepsilon}\left[\frac{(1+z)^{pr}}{z^{p+1}}\left(1-\frac{1}{(1+z)^r}\right)^{p+1}\right]dz+{p+r-1 \choose p}=\\
\vspace{-0.3cm}\\
\displaystyle\frac{1}{2\pi i}\int_{|z|=\varepsilon}\frac{((1+z)^r-1)^{p+1}}{(1+z)^rz^{p+1}}dz+{p+r-1 \choose p}={p+r-1 \choose p}
\end{array}
$$
as required. In the above calculation we utilized the facts that
$$
\displaystyle\int_{|z|=\varepsilon}\left[\frac{(1+z)^{pr}}{z^{p+1}}\,\frac{1}{(1+z)^{mr}}\right]dz=0 \quad \hbox{for $\displaystyle\left[\frac{p(r-1)}{r}\right]<m<p+1$}
$$
and
$$
\displaystyle\frac{1}{2\pi i}\int_{|z|=\varepsilon}\left[\frac{(1+z)^{pr}}{z^{p+1}}\,\frac{1}{(1+z)^{(p+1)r}}\right]dz=(-1)^p{p+r-1 \choose p}.
$$
Note that identity (\ref{proved2}) appears in \cite[formula (3.113)]{G}.

\subsection{Identity (\ref{identitytosolve3})} We assume that $n\ge 5$ and show that (\ref{identitytosolve3}) holds for all $m=5,\dots,n+1$ by induction on $m$. For $m=5$ both sides of (\ref{identitytosolve3}) are equal to $\displaystyle-4{n+1\choose 5}$, thus we assume that $6\le m\le n+1$. Using the induction hypothesis, we write the left-hand side of (\ref{identitytosolve3}) as
$$
\makebox[250pt]{$\begin{array}{l}
\displaystyle\delta_{m-2}-\hspace{-0.5cm}\sum_{\scalebox{0.5}{$\begin{array}{c}r+s=m-2,\\
r\ge 2,s\ge 1
\end{array}$}}{n+r-1\choose n-1}\delta_s+\\
\vspace{-0.3cm}\\
\displaystyle n\sum_{\ell=2}^{m-4}(-1)^{\ell}\hspace{-0.5cm}\sum_{\scalebox{0.5}{$\begin{array}{c}r_1+\dots+r_{\ell-2}+r_{\ell}+s=m-3,\\
r_1,\dots,r_{\ell-2}\ge 1,r_{\ell}\ge 2,s\ge 1
\end{array}$}}\hspace{-0.3cm}{n+r_1-1\choose n-1}\cdots{n+r_{\ell-2}-1\choose n-1}\times\\
\vspace{-0.3cm}\\
\displaystyle{n+r_{\ell}-1\choose n-1}\delta_s+{n+1\choose 2}\sum_{\ell=2}^{m-5}(-1)^{\ell}\hspace{-0.5cm}\sum_{\scalebox{0.5}{$\begin{array}{c}r_1+\dots+r_{\ell-2}+r_{\ell}+s=m-4,\\
r_1,\dots,r_{\ell-2}\ge 1,r_{\ell}\ge 2,s\ge 1
\end{array}$}}\hspace{-0.3cm}{n+r_1-1\choose n-1}\cdots\times\\
\end{array}$}
$$
\begin{equation}
\makebox[250pt]{$\begin{array}{l}
\displaystyle{n+r_{\ell-2}-1\choose n-1}\displaystyle{n+r_{\ell}-1\choose n-1}\delta_s+\dots+{n+m-6\choose m-5}{n+1\choose 2}\delta_1=\\
\vspace{-0.3cm}\\
\displaystyle \delta_{m-2}-\hspace{-0.3cm}\sum_{\scalebox{0.5}{$\begin{array}{c}r+s=m-2,\\
r\ge 2,s\ge 1
\end{array}$}}{n+r-1\choose n-1}\delta_s+\sum_{\rho=1}^{m-5}{n+\rho-1\choose n-1}\delta_{m-2-\rho}+\\
\vspace{-0.3cm}\\
\hspace{-0.0cm}\displaystyle (-1)^{m+1}\sum_{p=1}^{m-5}(-1)^p{n+1\choose m-p}(m-p-1) {n+p-1\choose p}=\\
\vspace{-0.3cm}\\
\displaystyle\delta_{m-2}-{n+m-4\choose m-3}\delta_1-{n+m-5\choose m-4}\delta_2+n\delta_{m-3}+\\
\vspace{-0.3cm}\\
\hspace{-0.0cm}\displaystyle (-1)^{m+1}\sum_{p=1}^{m-5}(-1)^p{n+1\choose m-p}(m-p-1){n+p-1\choose p}.
\end{array}$}\label{aux7778}
\end{equation}

To find the sum $\sum_{p=1}^{m-5}(-1)^p{n+1\choose m-p}(m-p-1){n+p-1\choose p}$, we will now show
\begin{equation}
\sum_{p=0}^{m}(-1)^p{n+1\choose m-p}{n+p-1\choose p}=0\label{aux788}
\end{equation}
and 
\begin{equation}
\sum_{p=0}^{m}(-1)^p{n+1\choose m-p}p{n+p-1\choose p}=0.\label{aux778}
\end{equation}

To obtain (\ref{aux788}), let us find the coefficient $C_m$ at $t^m$ in the expression\linebreak
$(1+t)^{n+1}\left(\sum_{\ell=0}^{\infty}(-t)^{\ell}\right)^n$.
An easy calculation shows that $C_m$ is exactly the left-hand side of (\ref{aux788}). On the other hand, since 
$$
(1+t)^{n+1}\left(\sum_{\ell=0}^{\infty}(-t)^{\ell}\right)^n=(1+t)^{n+1}\frac{1}{(1+t)^n}=1+t,
$$
we have $C_m=0$, which establishes (\ref{aux788}). Next, to obtain (\ref{aux778}), let us find the coefficient $C_m'$ at $t^m$ in the expression
$(1+t)^{n+1} \,t\left[\left(\sum_{\ell=0}^{\infty}(-t)^{\ell}\right)^n\right]'$.
It is easy to see that $C_m'$ is the left-hand side of (\ref{aux778}). On the other hand, 
$$
\begin{array}{l}
\displaystyle(1+t)^{n+1}\, t\left[\left(\sum_{\ell=0}^{\infty}(-t)^{\ell}\right)^n\right]'=(1+t)^{n+1}\, t\, \left[\frac{1}{(1+t)^n}\right]'=\\
\vspace{-0.3cm}\\
\displaystyle\hspace{3cm}-n(1+t)^{n+1}\, t\, \frac{1}{(1+t)^{n+1}}=-nt,
\end{array}
$$
hence we have $C_m'=0$, which establishes (\ref{aux778}).

Identities (\ref{aux788}), (\ref{aux778}) yield
$$
\sum_{p=0}^m(-1)^p{n+1\choose m-p}(m-p-1){n+p-1\choose p}=0.
$$
Therefore, (\ref{aux7778}) implies that the left-hand side of (\ref{identitytosolve3}) is equal to
$$
\begin{array}{l}
\displaystyle\delta_{m-2}-{n+m-4\choose m-3}\delta_1-{n+m-5\choose m-4}\delta_2+n\delta_{m-3}+\\
\end{array}
$$
\begin{equation}
\begin{array}{l}
\displaystyle(-1)^m\left[(m-1){n+1\choose m}+3(-1)^{m}{n+m-5\choose m-4}{n+1\choose 4}+\right.\\
\vspace{-0.3cm}\\
\displaystyle\left.2(-1)^{m+1}{n+m-4\choose m-3}{n+1\choose 3}+(-1)^{m}{n+m-3\choose m-2}{n+1\choose 2}+\right.\\
\vspace{-0.3cm}\\
\displaystyle\left.(-1)^{m+1}{n+m-1\choose m}\right]=(-1)^m(m-1){n+1\choose m}+\delta_{m-2}-\\
\vspace{-0.3cm}\\
\displaystyle{n+m-4\choose m-3}\delta_1+n\delta_{m-3}-2{n+m-4\choose m-3}{n+1\choose 3}+\\
\vspace{-0.3cm}\\
\displaystyle{n+m-3\choose m-2}{n+1\choose 2}-{n+m-1\choose m}.
\end{array}\label{aux8888}
\end{equation}
Using (\ref{deltas}) it is easy to show
$$
\begin{array}{l}
\displaystyle\delta_{m-2}-{n+m-4\choose m-3}\delta_1+n\delta_{m-3}-2{n+m-4\choose m-3}{n+1\choose 3}+\\
\vspace{-0.3cm}\\
\displaystyle\hspace{4cm}{n+m-3\choose m-2}{n+1\choose 2}-{n+m-1\choose m}=0,
\end{array}
$$
hence (\ref{aux8888}) yields that the left-hand side of (\ref{identitytosolve3}) is equal to $\displaystyle(-1)^m(m-1){n+1\choose m}$ as required.

\end{document}